\documentclass[11pt]{amsart}

\usepackage[T1]{fontenc}
\usepackage[utf8]{inputenc}
\usepackage{lmodern}
\usepackage{amsmath,amssymb,amsthm,mathtools}
\usepackage{enumitem}
\usepackage[hidelinks]{hyperref}
\usepackage[margin=1.15in]{geometry}

\newtheorem{theorem}{Theorem}[section]
\newtheorem{proposition}[theorem]{Proposition}
\newtheorem{lemma}[theorem]{Lemma}
\newtheorem{corollary}[theorem]{Corollary}
\newtheorem{claim}[theorem]{Claim}
\theoremstyle{definition}
\newtheorem{definition}[theorem]{Definition}

\newcommand{\Cay}{\operatorname{Cay}}
\newcommand{\girth}{\operatorname{girth}}
\newcommand{\diam}{\operatorname{diam}}

\title[Obstructions to coarse universality]{Obstructions to coarse universality for finitely generated groups}
\author{Robin Tucker-Drob}
\date{\today}

\begin{document}

\begin{abstract}
We prove that, for every bounded-degree graph $\Lambda$ admitting a finitely cobounded coarse quasi-action by a group, there is a finitely generated group which does not coarsely embed into $\Lambda$.
More generally, for every countable family $(\Lambda_i)$ of such graphs,
there is a finitely generated group that does not coarsely embed into any $\Lambda_i$.
This resolves two conjectures of Simon Thomas: neither a universal Cayley graph nor a universal quasi-isometry class of finitely generated groups exists.
As another consequence, we show that no locally compact second countable group coarsely contains every finitely generated group.

The proof uses an exponential upper bound on the number of finite graphs admitting an $(L,M)$-regular map into $\Lambda$, together with a superexponential supply of high-girth $3$-regular graphs, 
yielding a sequence of finite high-girth obstruction graphs. 
A graphical small-cancellation labeling, using a variation of Osajda's labeling theorem following Esperet and Giocanti, 
then realizes this sequence isometrically inside the Cayley graph of a finitely generated group.
\end{abstract}

\maketitle

\tableofcontents

\section{Introduction}

Simon Thomas conjectured that there is no universal Cayley graph and no universal quasi-isometry class of finitely generated groups \cite{Thomas}.
Here, a graph $\Lambda$ is a \emph{universal Cayley graph} if, for every finitely generated group $K$, there is a finitely generated group $G$ with finite generating set $S$ such that $K$ embeds as a subgroup of $G$ and $\Cay(G,S)\cong \Lambda$ as unlabeled graphs.
Similarly, a \emph{universal quasi-isometry class} is a quasi-isometry class $\mathcal Q$ of finitely generated groups such that every finitely generated group embeds as a subgroup of some member of $\mathcal Q$.
As Thomas observed, a universal Cayley graph would give a universal quasi-isometry class, so the nonexistence of the latter implies the nonexistence of the former.

We prove both of Thomas's conjectures by showing, more generally, that no bounded-degree graph admitting a finitely cobounded coarse quasi-action can coarsely contain every finitely generated group.

We call a coarse quasi-action \emph{finitely cobounded} if the union of finitely many quasi-orbits is coarsely dense (Definition~\ref{def:coarse-quasi-action}). 
For an honest action on a graph, this is the coarse analogue of having finitely many vertex orbits.

\begin{theorem}\label{thm:main-coarse}
Let $(\Lambda_i)_{i=1}^\infty$ be a sequence of bounded-degree graphs, each admitting a finitely cobounded coarse quasi-action by some group.
Then there exists a finitely generated group which does not coarsely embed into any $\Lambda_i$.
\end{theorem}

Applied to Cayley graphs, Theorem~\ref{thm:main-coarse} implies Thomas's two conjectures.

\begin{corollary}\label{cor:intro-consequences} There is no universal quasi-isometry class of finitely generated groups, and there is no universal Cayley graph.
\end{corollary}

The formulation of Theorem \ref{thm:main-coarse} in terms of finitely cobounded coarse quasi-actions also yields an extension to locally compact groups.
To state this extension, we recall that, by Struble's theorem, every locally compact second countable group $G$ admits a proper left-invariant compatible metric \cite{Struble}, and any two such metrics on $G$ are coarsely equivalent \cite{HaagerupPrzybyszewska}.
We regard each such group as a coarse space using any one of these metrics.

\begin{corollary}\label{cor:lcsc}
For every sequence $(G_i)_{i=1}^\infty$ of locally compact second countable groups, there is a finitely generated group which does not coarsely embed into any $G_i$.
In particular, no locally compact second countable group coarsely contains every finitely generated group.    
\end{corollary}

Because every discrete subgroup inclusion is a coarse embedding, Corollary~\ref{cor:lcsc} implies the following topological-group-theoretic fact, resolved here via coarse geometry:  \emph{there does not exist a locally compact second countable group that contains a discrete copy of every finitely generated group}.

The short proofs of the assertions in Corollaries~\ref{cor:intro-consequences} and \ref{cor:lcsc} from Theorem~\ref{thm:main-coarse} are given in Section~\ref{sec:consequences}.
We also note that the existence statements in Theorem~\ref{thm:main-coarse} and Corollary~\ref{cor:lcsc} admit $2^{\aleph_0}$ pairwise non-quasi-isometric witnesses; see Proposition~\ref{prop:free-product-amplification}.

In this article we use \emph{regular maps} (Definition \ref{def:regular}) as the bridge between coarse embeddings and a counting argument for finite graphs.
Regular maps are a large-scale weakening of graph embeddings, appearing for example in \cite{BST,Hume,HMT}.

A bird's-eye view of the proof of Theorem~\ref{thm:main-coarse} is as follows.
First fix one bounded-degree target graph $\Lambda$ admitting a finitely cobounded coarse quasi-action.
Every coarse embedding of a bounded-degree graph gives an $(L,M)$-regular map for suitable constants $L$ and $M$ (Lemma~\ref{lem:coarse-to-regular}).
Using the coarse quasi-action, each $(L,M)$-regular map into $\Lambda$ can be postcomposed so that its image meets a fixed finite subset of $V(\Lambda)$, at the cost of replacing $(L,M)$ by constants depending only on $L,M$ and the coarse quasi-action.
Together with the degree bound on $\Lambda$, this implies that, for fixed constants $L$ and $M$, the number of isomorphism classes of connected graphs on $s$-many vertices admitting an $(L,M)$-regular map into $\Lambda$ is at most exponential in $s$ (Lemma~\ref{lem:exp-bound}).
A superexponential supply of high-girth $3$-regular connected graphs due to Linial and Simkin \cite{LinialSimkin} (Theorem~\ref{thm:LS-high-girth-count}) then gives a finite high-girth obstruction graph for this target $\Lambda$ and these constants $L$ and $M$ (Proposition~\ref{prop:regular-obstruction}).
Diagonalizing over both the target graphs $\Lambda_i$ and the regular-map constants produces a single obstruction sequence for the whole family $(\Lambda_i)$ (Corollary~\ref{cor:choose-diagonal-graphs}).
A graphical small-cancellation labeling (Lemma~\ref{lem:labeling}), obtained via Esperet and Giocanti's proof \cite{EsperetGiocanti} of Osajda's labeling theorem \cite{Osajda}, then realizes this sequence isometrically inside the Cayley graph of a single finitely generated group (Theorem~\ref{thm:graphical-small-cancellation-embedding}).

\medskip

\noindent{\bf Acknowledgments.} I thank Simon Thomas for asking the questions that motivated this work and for pointing out the error in an earlier proposed argument.

This work was partially supported by NSF grant DMS-2246684.

\medskip

\noindent{\bf Disclosure of generative-AI tool use.}
In the course of researching and writing this article, I used the generative AI systems Claude, Gemini, and ChatGPT across several model versions,
the most recent being Claude Opus 4.8, Gemini 3.1 Pro, and GPT-5.5 Pro.
The systems were used for exploratory mathematical conversations, literature searches, and editing suggestions.
Many of the mathematical arguments they proposed were wrong; a few contained correct ideas that are used here in revised form.

At the outset of this project, I was familiar with neither the graphical small-cancellation results of Osajda \cite{Osajda} nor the high-girth graph count of Linial and Simkin \cite{LinialSimkin}.
ChatGPT suggested exploring a combination of Osajda's graphical small-cancellation labelings with a counting obstruction.
In a later conversation, it brought the Linial--Simkin high-girth counting result to my attention.
I developed the arguments presented here from these starting points.
Any errors are my own.

\section{Preliminaries}

Throughout, finitely generated groups are equipped with word metrics. 
Graphs are not assumed connected unless that assumption is stated explicitly. 
We equip graphs with the extended graph metric, so the distance between vertices in distinct components is $\infty$.
All graphs are simple and undirected.
When discussing labelings and graphical small-cancellation presentations in Sections~\ref{sec:labeling-lemma} and \ref{sec:realization}, we formally regard each undirected edge as a pair of oppositely directed edges.

\subsection{Coarse embeddings and regular maps}

\begin{definition}
A map $f:X\to Y$ from a metric space $X$ to an extended metric space $Y$ is a \emph{coarse embedding} if there are nondecreasing functions $\rho_-,\rho_+:[0,\infty)\to[0,\infty)$ with $\rho_-(t)\to\infty$ as $t\to\infty$ and
\[
        \rho_-(d_X(x,x'))
        \leq
        d_Y(f(x),f(x'))
        \leq
        \rho_+(d_X(x,x'))
\]
for all $x,x'\in X$.

We say that $Y$ \emph{coarsely contains} $X$ if there exists a coarse embedding from $X$ into $Y$.
\end{definition}

\begin{definition}\label{def:regular}
Let $X$ and $Y$ be graphs.  
Given constants $L\geq 0$ and $M\geq 1$, a map $f:V(X)\to V(Y)$ is called \emph{$(L,M)$-regular} if
\begin{enumerate}[label=(\roman*)]
\item $f$ is $L$-Lipschitz, i.e., $d_Y(f(x_0),f(x_1))\leq Ld_X(x_0,x_1)$ for all vertices $x_0,x_1\in V(X)$;
\item every fiber has size at most $M$:
\[
        |f^{-1}(y)|\leq M
        \qquad\text{for all }y\in Y.
\]
\end{enumerate}
\end{definition}

\begin{lemma}\label{lem:coarse-to-regular}
Let $X$ be a connected graph of maximum degree at most $\Delta <\infty$, and let $Y$ be a graph.
If $f:V(X)\to V(Y)$ is a coarse embedding, then there are constants $L,M<\infty$ such that $f$ is $(L,M)$-regular.
\end{lemma}

\begin{proof}
Let $\rho_-$ and $\rho_+$ be functions witnessing that $f$ is a coarse embedding.
Since $X$ is a graph, and $d_Y(f(x_0),f(x_1))\leq \rho_+(1)$ for adjacent vertices $x_0$ and $x_1$, it follows that $f$ is $L$-Lipschitz for $L=\rho_+(1)$.

Choose $R$ such that $\rho_-(R)>0$.
If $f(x_0)=f(x_1)$, then $\rho_-(d_X(x_0,x_1))\leq 0$, and hence $d_X(x_0,x_1)<R$.
Thus, each fiber of $f$ lies in an $R$-ball of $X$.
Since $X$ has maximum degree at most $\Delta$, there is a uniform bound $M=M(\Delta ,R)$ on the size of such a ball.
\end{proof}

\subsection{Finitely cobounded coarse quasi-actions}

We use the notion of coarse quasi-action from \cite[Definition~2.7]{BeckhardtGoldfarb}, adapted here to extended metric spaces.
Informally, a coarse quasi-action is an action by uniformly controlled coarse self-equivalences, with the action axioms holding up to uniformly bounded error.

\begin{definition}\label{def:coarse-quasi-action}
Let $G$ be a group and $Z$ an extended metric space.  
A \emph{coarse quasi-action} of $G$ on $Z$ is a family
$\sigma=(\sigma_g)_{g\in G}$ of self-maps of $Z$ for which there exist a nondecreasing function
$\omega:[0,\infty]\to[0,\infty]$ with $\omega(\infty)=\infty$ and $\omega(t)<\infty$ for $t<\infty$, and a constant $C<\infty$ such that:
\begin{enumerate}[label=(\roman*)]
\item $d(\sigma_gx,\sigma_gy)\leq \omega(d(x,y))$ for every $g\in G$ and $x,y\in Z$;
\item $d(\sigma_{1_G}x,x)\leq C$ for every $x\in Z$;
\item $d(\sigma_g\sigma_h x,\sigma_{gh}x)\leq C$ for every $g,h\in G$ and $x\in Z$.
\end{enumerate}

A coarse quasi-action is \emph{finitely cobounded} if there are $z_1,\dots,z_r\in Z$ and $R<\infty$ such that, for every $z\in Z$, there are $g\in G$ and $1\leq i\leq r$ with $d(z,\sigma_gz_i)\leq R$.
It is called \emph{cobounded} if one may take $r=1$.
\end{definition}

The axioms imply that
\[
        d(\sigma_{g^{-1}}\sigma_gx,x)\leq 2C
        \qquad\text{and}\qquad
        d(\sigma_g\sigma_{g^{-1}}x,x)\leq 2C
\]
for every $g\in G$ and $x\in Z$.  Thus each $\sigma_g$ is a coarse equivalence with coarse inverse $\sigma_{g^{-1}}$.  

A \emph{quasi-action} is a coarse quasi-action for which $\omega$ may be chosen of the form $\omega(t)=Lt+A$ for some $L\geq 1$ and $A\geq 0$.
In this case, each $\sigma_g$ is a quasi-isometry, with constants independent of $g$, and quasi-inverse $\sigma_{g^{-1}}$.
On a graph, every coarse quasi-action is a quasi-action, so the two notions coincide for the graph targets considered in this article.  
We use the coarse terminology throughout.

\section{Obstruction graphs via regular maps}

\subsection{Counting graphs with regular maps into \texorpdfstring{$\Lambda$}{Lambda}}

We use the following elementary estimate.

\begin{lemma}\label{lem:connected-subsets}
Let $Y$ be a rooted graph of maximum degree at most $\Delta$, where $1\leq \Delta<\infty$.
The number of connected subsets of
$V(Y)$ of size $m$ that contain the root is at most $\Delta^{2m-2}$.
\end{lemma}

\begin{proof}
Let $W\subseteq V(Y)$ be a connected subset of size $m$ containing the root, and fix a spanning tree for the induced subgraph on $W$.
Starting at the root, one can perform a ``depth-first'' walk through the entire spanning tree, ending back at the root, in such a way
that every edge is crossed exactly twice: once while moving away from
the root and once while returning toward the root.  
Such a walk has length exactly $2m-2$ and visits every vertex of $W$.

A walk of length $2m-2$ starting at the root has at most $\Delta$-many choices at each step, so there are at most $\Delta^{2m-2}$ such walks.  
Moreover, the walk determines its set of visited
vertices.  
Hence, the number of connected subsets of $V(Y)$ of size $m$ containing the root is at most $\Delta^{2m-2}$.
\end{proof}

The following exponential upper bound is the main new combinatorial observation of the paper.
It will later be compared with a superexponential supply of high-girth graphs.

\begin{lemma}[Exponential upper bound for graphs admitting $(L,M)$-regular maps into $\Lambda$]\label{lem:exp-bound}
Let $\Lambda$ be a bounded-degree graph admitting a finitely cobounded coarse quasi-action by a group, and let $L,M\geq 1$ be integers.  
There is a constant $A=A(L,M,\Lambda)<\infty$ such that, for every $s\geq 1$, there are at most $A^s$-many isomorphism classes of connected graphs on $s$-many vertices that admit an $(L,M)$-regular map into $\Lambda$.
\end{lemma}

\begin{proof}
The hypothesis that $\Lambda$ admits a finitely cobounded coarse quasi-action is essential for the proof:
it allows us to postcompose each regular map with one of the quasi-action maps so that its image meets a fixed finite set.
We then count the possible images and, for each such image $W$, count all possible subgraphs of the lexicographic product of $\Lambda ^{\leq L_0}|_W$ with the complete graph $K_{M_0}$, for suitably chosen $(L_0,M_0)$.

We may assume $V(\Lambda)\neq\varnothing$.
We first fix some notation and parameters.
Fix a finitely cobounded coarse quasi-action $\sigma=(\sigma_g)_{g\in G}$ of a group $G$ on $\Lambda$, along with $\omega$ and $C<\infty$ as in Definition~\ref{def:coarse-quasi-action}. 
Fix also vertices $z_1,\dots,z_r\in V(\Lambda)$ and $R<\infty$ witnessing finite coboundedness.  
Define
\[
        L_0\coloneqq \max\{1,\lceil\omega(L)\rceil\} ,
        \qquad
        R_0\coloneqq \omega(R)+2C, \qquad \text{and}\qquad Z_0\coloneqq \bigcup_{i=1}^r B_\Lambda(z_i,R_0),
\]
where $B_\Lambda (\cdot , R_0)$ denotes the closed ball of radius $R_0$. 
The set $Z_0$ is finite since $\Lambda$ has bounded degree.

Choose $N\in\mathbb{N}$ such that every ball of radius $2C$ in $\Lambda$
has cardinality at most $N$, and let $M_0\coloneqq MN$.

Let $\Lambda^{\leq L_0}$ be the graph with vertex set $V(\Lambda)$ in which
two distinct vertices are adjacent if their $\Lambda$-distance is at most
$L_0$. 
Because $\Lambda$ has bounded degree, $\Lambda^{\leq L_0}$ also has bounded degree.
Fix an integer $\Delta_0\geq 1$ such that every vertex of $\Lambda^{\leq L_0}$ has degree at
most $\Delta_0$. 

With these parameters fixed, let $X$ be a connected graph on $s$-many vertices admitting an $(L,M)$-regular map into $\Lambda$, say $f:V(X)\to V(\Lambda)$.
Choose $x_0\in V(X)$.  
By finite coboundedness, there exist $g\in G$ and $1\leq i\leq r$ such that $d_\Lambda(f(x_0),\sigma_gz_i)\leq R$.
Define $h:V(X)\to V(\Lambda)$ by $h=\sigma_{g^{-1}}\circ f$.  
\begin{claim}\label{claim:L0M0}
$h$ is an $(L_0,M_0)$-regular map into $\Lambda$, and $h(V(X))$ meets $Z_0$.
In particular $h(V(X))$ is connected in $\Lambda^{\leq L_0}$.
\end{claim}
\begin{proof}[Proof of Claim~\ref{claim:L0M0}:]
Since
\begin{align*}
        d_\Lambda(h(x_0),z_i)
        &\leq d_\Lambda(\sigma_{g^{-1}}f(x_0),\sigma_{g^{-1}}\sigma_gz_i)
        +d_\Lambda(\sigma_{g^{-1}}\sigma_gz_i,\sigma_{1_G}z_i)+d_\Lambda(\sigma_{1_G}z_i,z_i)\\
        &\leq \omega (R) + 2C = R_0,
\end{align*}
the set $h(V(X))$ meets $Z_0$.

For fixed $w\in V(\Lambda)$, if $h(x)=w$ then
\[
        d_\Lambda(f(x),\sigma_gw)
        =d_\Lambda(f(x),\sigma_g\sigma_{g^{-1}}f(x))
        \leq 2C.
\]
Thus $h^{-1}(w)\subseteq f^{-1}\bigl(B_\Lambda(\sigma_gw,2C)\bigr)$.
Since every fiber of $f$ has cardinality at most $M$ and every ball of radius $2C$ has cardinality at most $N$, every fiber of $h$ has cardinality
at most $MN=M_0$.

If $x$ and $y$ are adjacent in $X$, then
\[
        d_\Lambda(h(x),h(y))
        \leq \omega(d_\Lambda(f(x),f(y)))
        \leq \omega(L)
        \leq L_0.
\]
Since $X$ is a graph, this implies that $h$ is $L_0$-Lipschitz.
Together with the preceding fiber bound, it follows that $h$ is an $(L_0,M_0)$-regular map into $\Lambda$.  Since $X$ is connected, $h(V(X))$ is connected in $\Lambda^{\leq L_0}$.\qedhere[Claim~\ref{claim:L0M0}]
\end{proof}

For a $\Lambda^{\leq L_0}$-connected subset $W\subseteq V(\Lambda)$, let $W[M_0]$ be the lexicographic product of the subgraph of $\Lambda^{\leq L_0}$ induced by $W$ with the complete graph on $\{0,\dots,M_0-1\}$: that is, replace each vertex $w\in W$ by a clique on $\{w\}\times\{0,\dots,M_0-1\}$, and whenever $v,w\in W$ are adjacent in the induced subgraph $\Lambda^{\leq L_0}|_W$, join the two corresponding cliques by a complete bipartite graph.    

Taking $W=h(V(X))$, since every fiber of $h$ has cardinality at most $M_0$, the graph $X$ is isomorphic to a (not necessarily induced) subgraph of $W[M_0]$.
Indeed, we may choose a map $V(X)\to \{ 0,\dots , M_0-1 \}$, $x\mapsto i_x$, that is injective on $h^{-1}(w)$ for each $w\in W$, and then realize $X$ as a subgraph of $W[M_0]$ via the map $x\mapsto(h(x),i_x)$.

So an upper bound on the total number of isomorphism classes of connected graphs on $s$-many vertices which admit an $(L,M)$-regular map into $\Lambda$ is given by 
\begin{equation}\label{eqn:subgraphsum}
\sum _W (\text{number of subgraphs of }W[M_0]),
\end{equation}
with $W$ ranging over all $\Lambda^{\leq L_0}$-connected subsets of $V(\Lambda)$ of cardinality at most $s$ which meet $Z_0$.

Given such a $W$ of cardinality $m$, we can bound the number of edges of $W[M_0]$ via
\[
        |E(W[M_0])|
        \leq
        m\binom{M_0}{2}+\frac{\Delta_0m}{2}M_0^2
        \leq C_0m
\]
for a constant $C_0=C_0(L,M,\Lambda)$.  Hence the total number of subgraphs of $W[M_0]$ is at most
\[
        2^{|V(W[M_0])|+|E(W[M_0])|}
        \leq 2^{mM_0+C_0m}
        \leq C_1^m
\]
for some constant $C_1=C_1(L,M,\Lambda)\geq 2$.

By Lemma~\ref{lem:connected-subsets}, for each fixed $m\leq s$ and each fixed $z\in Z_0$, there are at most $(\Delta_0^2)^m$-many $\Lambda^{\leq L_0}$-connected subsets of cardinality $m$ containing $z$.  
Thus, those $W$'s with $|W|=m$ contribute at most 
\[
        |Z_0|(\Delta_0^2C_1)^m
\]
to the sum \eqref{eqn:subgraphsum}.
Summing over $1\leq m\leq s$, the total number of isomorphism classes is at most
\[
        \sum_{m=1}^s |Z_0|(\Delta_0^2C_1)^m
        \leq
        |Z_0|\frac{\Delta_0^2C_1}{\Delta_0^2C_1-1}(\Delta_0^2C_1)^s
        \leq A^s
\]
for a sufficiently large constant $A=A(L,M,\Lambda)$.
\end{proof}

\subsection{Superexponentially many high-girth graphs}
Our later small-cancellation realization requires a sequence of high-girth obstruction graphs.
Such a sequence will be chosen from the following superexponential supply of high-girth graphs.

\begin{theorem}[Linial--Simkin {\cite[Theorem~1.2]{LinialSimkin}}]\label{thm:LS-high-girth-count}
Fix an integer $k\geq 3$ and a real number $c\in(0,1)$.  
There are constants $a=a_{k,c}>0$ and $s_{k,c}\in\mathbb{N}$ such that, for every even $s\geq s_{k,c}$, there are at least $(a s)^{ks/2}$-many Hamiltonian, and hence connected, simple $k$-regular graphs on the vertex set $\{0,\dots,s-1\}$ with girth at least $c\log_{k-1}s$.
\end{theorem}

The theorem is stated in \cite[Theorem~1.2]{LinialSimkin} without the word ``Hamiltonian,'' but the graphs counted in the proof are produced by a process that starts from a Hamilton cycle and only adds edges; this construction is described in the introduction and analyzed in Section~5 of \cite{LinialSimkin}.
Thus, every graph counted in that argument contains the initial Hamilton cycle.

\begin{proposition}[High-girth obstruction to regular maps]\label{prop:regular-obstruction}
Let $\Lambda$ be a bounded-degree graph admitting a finitely cobounded coarse quasi-action by a group.
For every $c\in(0,1)$ and every pair of integers $L,M\geq 1$, for all sufficiently large even $s$ there exists a connected simple $3$-regular graph $X$ on $s$-many vertices such that
\[
\girth(X)\geq c\log_2 s
\]
and $X$ admits no $(L,M)$-regular map into $\Lambda$.
\end{proposition}

\begin{proof}
Fix $c\in(0,1)$ and $L,M\geq 1$, and let $A=A(L,M,\Lambda)$ be given by Lemma~\ref{lem:exp-bound}.  Applying Theorem~\ref{thm:LS-high-girth-count} with $k=3$, and writing $a=a_{3,c}$, we see that for all sufficiently large even $s$ there are at least $(a s)^{3s/2}$-many Hamiltonian, and hence connected, simple $3$-regular graphs on the vertex set $\{0,\dots,s-1\}$ with girth at least $c\log_2s$.

Each isomorphism class has at most $s!$-many representatives on $\{0,\dots,s-1\}$, and thus the number of isomorphism classes of such graphs is at least
\[
        \frac{(a s)^{3s/2}}{s!}
        \geq
        \frac{(a s)^{3s/2}}{s^s}
        =
        (a^{3/2}s^{1/2})^s.
\]
By Lemma~\ref{lem:exp-bound}, the number of isomorphism classes of connected graphs on $s$-many vertices admitting an $(L,M)$-regular map into $\Lambda$ is at most $A^s$.  
For sufficiently large $s$, the preceding high-girth lower bound exceeds $A^s$, giving the conclusion.
\end{proof}

\begin{corollary}[Sequence of obstruction graphs]\label{cor:choose-diagonal-graphs}
Let $(\Lambda_i)_{i=1}^\infty$ be a sequence of bounded-degree graphs, each admitting a finitely cobounded coarse quasi-action by a group, and fix $\lambda>0$.
There exists a sequence $(X_j)_{j=1}^\infty$ of finite, connected, simple, $3$-regular graphs such that, writing
$g_j\coloneqq \girth(X_j)$:
\begin{enumerate}[label=(\alph*)]
\item for every $1\leq i\leq j$, the graph $X_j$ admits no $(j,j)$-regular map into $\Lambda_i$;
\item the integers $\gamma_j\coloneqq \lfloor\lambda g_j\rfloor$ satisfy
$1<\gamma_j<\gamma_{j+1}$ for all $j\geq 1$;
\item there is a constant $B<\infty$ independent of $j$ such that $|E(X_j)|\leq B^{g_j}$.
\end{enumerate}
\end{corollary}

\begin{proof}
For each $i\geq 1$, fix a group $G_i$ with a finitely cobounded coarse quasi-action on $\Lambda_i$.
For each $j\geq 1$, let $\Lambda^{(j)}\coloneqq \bigsqcup_{i=1}^j\Lambda_i$ and $G^{(j)}\coloneqq G_1\times\cdots\times G_j$.
Then $\Lambda^{(j)}$ has bounded degree, and $G^{(j)}$ has a finitely cobounded coarse quasi-action on $\Lambda^{(j)}$: namely, the coarse quasi-action whose restriction to each $\Lambda_i$ factors through the $i$-th coordinate projection $G^{(j)}\to G_i$ to the given coarse quasi-action of $G_i$ on $\Lambda_i$.

Fix $c\in(0,1)$.  We define the sequence $(X_j)_{j=1}^\infty$ recursively. 
Let $\gamma_0=1$ and, for $j\geq 1$, having defined $(X_i)_{1\leq i<j}$, choose an even integer $s_j$ sufficiently large that Proposition~\ref{prop:regular-obstruction} applies to $\Lambda^{(j)}$ with $(L,M)=(j,j)$ and
\[
\lfloor\lambda c\log_2s_j\rfloor>\gamma_{j-1}.
\]
By the choice of $s_j$, we may find a connected simple $3$-regular graph $X_j$ on $s_j$-many vertices with $g_j\coloneqq \girth(X_j)\geq c\log_2s_j$ and admitting no $(j,j)$-regular map into $\Lambda^{(j)}$, and hence none into any of $\Lambda_1,\dots,\Lambda_j$. 
Then $\gamma_j=\lfloor\lambda g_j\rfloor\geq \lfloor\lambda c\log_2s_j\rfloor>\gamma_{j-1}$, proving (a) and (b).

Finally, $g_j\geq c\log_2s_j$ implies
$s_j\leq 2^{g_j/c}$.  Since $g_j\geq 1$ and $X_j$ is $3$-regular,
\[
|E(X_j)|=\frac{3}{2}s_j
\leq 2^{g_j}2^{g_j/c}
=\bigl(2^{1+1/c}\bigr)^{g_j}.
\]
Thus (c) holds with $B=2^{1+1/c}$.
\end{proof}

\section{A labeling lemma after Osajda and Esperet--Giocanti}
\label{sec:labeling-lemma}

Osajda proved in \cite{Osajda} that a sequence $(\Theta_n)_{n=1}^\infty$ of finite connected graphs of uniformly bounded degree, growing girth, and uniformly bounded diameter-to-girth ratio admits a finite-alphabet small-cancellation labeling.
Esperet and Giocanti \cite{EsperetGiocanti} reproved this by a counting argument which accommodates a smaller labeling alphabet.

Although \cite[Theorem 3.1]{EsperetGiocanti} is stated under the same diameter-to-girth hypothesis, the role of that hypothesis in the counting proof is to give a uniform exponential bound on the total number of nonbacktracking paths of length $\gamma_n=\lfloor\lambda\girth(\Theta_n)\rfloor$.
Section~4 of \cite{EsperetGiocanti} explains that an edge bound in terms of girth can be used instead, since such a bound gives the same path-count estimate.
Lemma~\ref{lem:labeling} records the resulting form needed below, replacing the hypothesis $\diam(\Theta_n)\leq A\girth(\Theta_n)$ by the exponential edge bound $|E(\Theta_n)|\leq B^{\girth(\Theta_n)}$.
In our application this bound will be supplied by Corollary~\ref{cor:choose-diagonal-graphs}(c).

\begin{definition}\label{def:small-cancellation-labeling}
Let $\Theta=\bigsqcup_n\Theta_n$ be a disjoint union of finite connected graphs, and let $S$ be a finite alphabet with a fixed-point-free involution $s\mapsto s^{-1}$.  
A \emph{labeling} of $\Theta$ by $S$ assigns a letter of $S$ to each directed edge so that reversing a directed edge inverts its letter.
For a directed path $P=e_1\cdots e_k$ in $\Theta$ the \emph{word read along $P$} is $\ell(P)\coloneqq \ell(e_1)\cdots\ell(e_k)$.  
The labeling is \emph{reduced} if no two distinct directed edges with the same origin vertex have the same label.

We say that a reduced labeling satisfies the \emph{$C'(\lambda)$ small-cancellation property} if, for every $n$, no word of length at least $\lambda\girth(\Theta_n)$ read along a nonbacktracking path in $\Theta_n$ occurs as the word read along any other directed nonbacktracking path in $\Theta$, including another path in the same component.
\end{definition}

\begin{lemma}[Labeling lemma, after Osajda \cite{Osajda} and Esperet--Giocanti \cite{EsperetGiocanti}]\label{lem:labeling}
Fix $\Delta,B<\infty$ with $\Delta\geq 3$ an integer and $B\geq 2$, and fix $0<\lambda\leq 1/6$.
Let $(\Theta_n)_{n=1}^\infty$ be a sequence of finite connected simple graphs of maximum degree at most $\Delta$ such that, letting
$g_n\coloneqq \girth(\Theta_n)$ and $\gamma_n\coloneqq \lfloor \lambda g_n\rfloor$,
\[
        g_n\to\infty,
        \qquad
        |E(\Theta_n)|\leq B^{g_n},
        \qquad
        1<\gamma_n<\gamma_{n+1}\text{ for every }n\geq 1.
\]
Then $\bigsqcup_{n=1}^\infty \Theta_n$ admits a reduced finite-alphabet labeling
satisfying the $C'(\lambda)$ small-cancellation property.
\end{lemma}

\begin{proof}
Let $p_n$ be the number of nonbacktracking directed paths of length $\gamma_n$ in $\Theta_n$.
Choosing the first directed edge and then the remaining $\gamma_n-1$ directed edges gives
\[
        p_n
        \leq 2|E(\Theta_n)|(\Delta-1)^{\gamma_n-1}
        \leq 2B^{g_n}\Delta^{\gamma_n}.
\]
Since $\gamma_n=\lfloor\lambda g_n\rfloor$ and $\gamma_n>1$, we have $g_n/\gamma_n \leq 2/\lambda$, and hence
\[
        p_n
        \leq 2B^{g_n}\Delta^{\gamma_n}
        \leq \bigl(2B^{2/\lambda}\Delta\bigr)^{\gamma_n}.
\]
Thus $p_n\leq C^{\gamma_n}$ for every $n$, where
$C\coloneqq 2B^{2/\lambda}\Delta$.

The conclusion now follows from the counting proof of
\cite[Theorem~3.1]{EsperetGiocanti}, with the substitution described in \cite[Section~4]{EsperetGiocanti}.
In their proof, the diameter-to-girth hypothesis enters through the estimate for the number of possible comparison paths; once that estimate is replaced, the remaining counting argument is unchanged.
Explicitly, in the proof of \cite[Theorem~3.1]{EsperetGiocanti}, replace the bound, coming from their equation~(6), on the number of choices for the path $Q$, by $p_i\leq C^{\gamma_i}$, and take $\alpha=2(\Delta-1)C$.
Then the estimate in their equation~(7) remains valid.  
Taking, as in their proof, an alphabet $S$ with $|S|$ even and $|S|\geq 2(\Delta-1)+13\alpha$, the remainder of their argument is unchanged, and produces the desired labeling of $\bigsqcup_{n=1}^\infty \Theta_n$ by $S$.
\end{proof}

\section{The small-cancellation realization}\label{sec:realization}

\subsection{Isometric embedding theorem}

We record the embedding theorem used below.
Let $0<\lambda\leq 1/6$, let
$\Theta=\bigsqcup_{n=1}^\infty\Theta_n$ be a disjoint union of finite connected
graphs, let $S=S_0\sqcup S_0^{-1}$ be a finite alphabet with fixed-point-free
involution $s\mapsto s^{-1}$, and let $\ell$ be a reduced labeling of $\Theta$
by $S$ satisfying the $C'(\lambda)$ small-cancellation property of
Definition~\ref{def:small-cancellation-labeling}.

A \emph{rooted directed simple cycle} in $\Theta$ is a simple cycle together
with a choice of initial vertex and direction.
Define the associated graphical small-cancellation group by the presentation
\[
        G(\Theta,\ell)
        \coloneqq
        \left\langle S_0\ \middle|\
        \ell(C),\ C\text{ is a rooted directed simple cycle in }\Theta
        \right\rangle .
\]
We write $S$ also for its image in $G(\Theta,\ell)$.
For each $n$, choose a base vertex $o_n\in V(\Theta_n)$.  
The labeling defines a
natural map
\[
        \iota_n:\Theta_n\to\Cay(G(\Theta,\ell),S)
\]
sending $o_n$ to the identity: if $P$ is a directed path from $o_n$ to
$v\in V(\Theta_n)$, then $\iota_n(v)$ is the element represented by $\ell(P)$.
The defining relators make this independent of the choice of $P$.

The following theorem is attributed to Gromov \cite{Gromov}. 
A detailed proof was given by Ollivier \cite[Theorem~1]{Ollivier}. 
We use a formulation given by Gruber \cite[Theorem~5.10]{Gruber}.

\begin{theorem}[Graphical small-cancellation embedding theorem]
\label{thm:graphical-small-cancellation-embedding}
With the notation and assumptions above, the map
\[        \iota_n:\Theta_n\to\Cay(G(\Theta,\ell),S)
\]
is an isometric embedding for every $n$.
\end{theorem}

\subsection{Proof of Theorem \ref{thm:main-coarse}}

\begin{proof}[Proof of Theorem~\ref{thm:main-coarse}]
Fix a sequence $(\Lambda_i)_{i=1}^\infty$ satisfying the hypothesis of Theorem~\ref{thm:main-coarse}, and fix $0<\lambda\leq 1/6$.
Choose the sequence $(X_j)_{j=1}^\infty$ of $3$-regular graphs from Corollary~\ref{cor:choose-diagonal-graphs} using this $\lambda$ and let $X\coloneqq\bigsqcup_{j=1}^\infty X_j$.
Thus, for every $1\leq i\leq j$, the graph $X_j$ admits no $(j,j)$-regular map into $\Lambda_i$.
By Lemma~\ref{lem:labeling}, there is a finite alphabet $S=S_0\sqcup S_0^{-1}$ with a fixed-point-free involution $s\mapsto s^{-1}$, and a reduced labeling $\ell$ of $X$ by $S$ satisfying the $C'(\lambda)$ small-cancellation property.

Let $K\coloneqq G(X,\ell)$ be the associated graphical small-cancellation group, which is finitely generated by $S_0$.
By Theorem~\ref{thm:graphical-small-cancellation-embedding}, the natural maps
\[
        \iota_j:X_j\to\Cay(K,S)
\]
are isometric embeddings.
It remains to show that the group $K$ does not coarsely embed into any of the graphs $\Lambda_i$.

Fix $i\geq 1$, and suppose toward a contradiction that there is a coarse embedding $f:\Cay(K,S)\to\Lambda_i$.
By Lemma~\ref{lem:coarse-to-regular}, since $\Cay(K,S)$ has bounded degree,
$f$ is $(L,M)$-regular for some integers $L,M\geq 1$.
Choose $j\geq \max\{i,L,M\}$.
Since $\iota_j$ is an isometric embedding, the composition $f\circ\iota_j:X_j\to \Lambda_i$ is $(L,M)$-regular, and hence $(j,j)$-regular.
This contradicts the choice of $X_j$.
\end{proof}

\section{Consequences}\label{sec:consequences}

\subsection{Coarse embeddings into locally compact second countable groups}

Recall that a locally compact second countable group is equipped with a canonical coarse structure, coming from any proper left-invariant compatible metric; see \cite{CornulierDeLaHarpe} for a comprehensive background on the coarse geometry of locally compact groups.

\begin{proposition}[Closed subgroup embeddings are coarse]\label{prop:closed-image-coarse}
Let $H$ and $G$ be locally compact second countable groups, and let
$\iota:H\to G$ be an injective continuous homomorphism with closed image.
Then $\iota$ is a coarse embedding.
In particular, every injective homomorphism between countable discrete groups is a coarse embedding.
\end{proposition}

\begin{proof}
By \cite{HofmannMorris}, the homomorphism $\iota$ is a topological isomorphism onto the closed subgroup $\iota(H)$.
A proper left-invariant compatible metric on $G$ then pulls back through $\iota$ to such a metric on $H$ which,
by \cite{HaagerupPrzybyszewska}, induces the canonical coarse structure on $H$.
Thus, $\iota$ is a coarse embedding.
\end{proof}

\begin{lemma}[Compactly generated exhaustion]\label{lem:compactly-generated-exhaustion}
Let $G$ be a locally compact second countable group equipped with a proper left-invariant compatible metric $d$.
For $n\geq 1$, put
\[
        Q_n\coloneqq\overline{B_d(1,n)}
        \qquad\text{and}\qquad
        G_n\coloneqq\langle Q_n\rangle.
\]
Then the following hold:
\begin{enumerate}[label=(\alph*)]
\item
$(G_n)_{n=1}^\infty$ is an increasing sequence of open compactly generated subgroups whose union is $G$.

\item
Let $d_{Q_n}^{\mathrm{word}}$ denote the word metric on $G_n$ associated to the compact generating set $Q_n$.
If a connected graph $X$ coarsely embeds into $(G,d)$, then $X$ coarsely embeds into $(G_n,d_{Q_n}^{\mathrm{word}})$ for some $n$.
\end{enumerate}
\end{lemma}

Here $d_{Q_n}^{\mathrm{word}}$ denotes the word metric on $G_n$ associated to $Q_n$, rather than a metric inherited from $G$.
Since $Q_n$ is compact and generates $G_n$, this metric is coarsely equivalent to every proper left-invariant compatible metric on $G_n$.
We use the word metric because it is the metric naturally modeled by the graphs used in the proof of Corollary~\ref{cor:lcsc} below.

\begin{proof}
Each $Q_n$ is a compact symmetric neighborhood of the identity.
Thus $G_n$ is open and compactly generated.
The sequence is increasing, and its union is $G$ because every element of $G$ lies in some $Q_n$.
This proves (a).

For (b), let $f:X\to G$ be a coarse embedding with associated maps $\rho_-$ and $\rho_+$, and fix $x_0\in V(X)$.
After composing with a left translation, assume that $f(x_0)=1$.
Choose an integer $n>\rho_+(1)$.
If $x$ and $y$ are adjacent, then $f(x)^{-1}f(y)\in Q_n$.
Since $X$ is connected, $f(V(X))\subseteq G_n$, and $f:X\to(G_n,d_{Q_n}^{\mathrm{word}})$ is $1$-Lipschitz.

For every integer $r\geq 0$, put
\[
        M_r\coloneqq\sup\{d(1,g):g\in Q_n^r\}<\infty.
\]
Suppose $d_{Q_n}^{\mathrm{word}}(f(x),f(y))\leq r$. 
Then $f(x)^{-1}f(y)\in Q_n^r$, and hence
\[
        \rho_-(d_X(x,y))
        \leq d(f(x),f(y))
        \leq M_r.
\]
Since $\rho_-(t)\to\infty$, there is $R_r<\infty$ such that $\rho_-(R_r)>M_r$, and hence $d_X(x,y)<R_r$.
Together with the fact that $f:X\to(G_n,d_{Q_n}^{\mathrm{word}})$ is $1$-Lipschitz, this shows that $f$ is a coarse embedding.
\end{proof}

\subsection{Proof of Corollaries~\ref{cor:intro-consequences} and~\ref{cor:lcsc}} We now prove the corollaries from the introduction.

\begin{proof}[Proof of Corollary~\ref{cor:lcsc}]
For each $i$, fix a proper left-invariant compatible metric $d_i$ on $G_i$.
Apply Lemma~\ref{lem:compactly-generated-exhaustion} to $(G_i,d_i)$, and write
$(G_{i,n})_{n=1}^\infty$ and $(Q_{i,n})_{n=1}^\infty$ for the resulting subgroups and compact generating sets.
By \cite[Theorem~3.2]{Salmi}, for each pair $(i,n)$ there is a connected bounded-degree graph $\Lambda_{i,n}$ on which $G_{i,n}$ admits a proper cobounded quasi-action.
In particular, $\Lambda_{i,n}$ admits a finitely cobounded coarse quasi-action.
By \cite[Lemma~2.2]{Salmi}, the graph $\Lambda_{i,n}$ is quasi-isometric to $(G_{i,n},d_{Q_{i,n}}^{\mathrm{word}})$.

Apply Theorem~\ref{thm:main-coarse} to the countable family $(\Lambda_{i,n})_{i,n\geq 1}$,
and let $K$ be the resulting finitely generated group.
Suppose toward a contradiction that, for some $i$, the group $K$ coarsely embeds into $(G_i,d_i)$.
Then Lemma~\ref{lem:compactly-generated-exhaustion}(b) would give a coarse embedding of $K$
into $(G_{i,n},d_{Q_{i,n}}^{\mathrm{word}})$ for some $n$.
Composing with a quasi-isometry $(G_{i,n},d_{Q_{i,n}}^{\mathrm{word}})\to\Lambda_{i,n}$ would contradict the choice of $K$.
Thus, for every $i$, the group $K$ does not coarsely embed into $G_i$.
\end{proof}

\begin{proof}[Proof of Corollary~\ref{cor:intro-consequences}]
Since two finitely generated groups with isomorphic Cayley graphs are quasi-isometric, it is enough to prove that there is no universal quasi-isometry class.
Suppose toward a contradiction that $\mathcal Q$ were a universal quasi-isometry class, and choose a finitely generated group $G\in\mathcal Q$.
Fix a finite generating set $S$ of $G$, and let $\Lambda\coloneqq\Cay(G,S)$.
Since $\Lambda$ is bounded degree and vertex-transitive, the constant sequence $\Lambda_i=\Lambda$ satisfies the hypotheses of Theorem~\ref{thm:main-coarse}.
Thus there is a finitely generated group $K$ which does not coarsely embed into $\Lambda$.
By universality of $\mathcal Q$, the group $K$ embeds as a subgroup of a finitely generated group $H\in\mathcal Q$.
By Proposition~\ref{prop:closed-image-coarse}, the inclusion $K\to H$ is a coarse embedding, and composing with a quasi-isometry $H\to G$ gives a coarse embedding of $K$ into $\Lambda$, a contradiction.
\end{proof}

\subsection{Continuum many witnesses}

The following is standard, but we could not find a reference, so we include the proof.

\begin{proposition}[Continuum many quasi-isometry classes of overgroups]\label{prop:free-product-amplification}
Let $K$ be a finitely generated group.
There exists a continuum-sized family $(H_i)_{i\in I}$ of pairwise non-quasi-isometric finitely generated groups,
each containing $K$ as a subgroup.

In particular, if a metric space $Y$ does not coarsely contain $K$,
then $Y$ does not coarsely contain any member of the family $(H_i)_{i\in I}$.
\end{proposition}

\begin{proof}
The product $K\times\mathbb{Z}^2$ is finitely generated and one-ended.
Let $(B_i)_{i\in I}$ be a continuum-sized family of pairwise non-quasi-isometric finitely generated groups that are each one-ended;
for example, this is provided by Grigorchuk's construction \cite{GrigorchukGrowth} of continuum many finitely generated groups of intermediate growth with pairwise distinct growth types.
By Papasoglu--Whyte \cite[Theorem~0.4]{PapasogluWhyte}, the free products $H_i\coloneqq (K\times\mathbb{Z}^2) * B_i$, for $i\in I$, are pairwise non-quasi-isometric, and they each contain $K$.
The final assertion follows because each inclusion $K\hookrightarrow H_i$ is a coarse embedding by Proposition~\ref{prop:closed-image-coarse}.
\end{proof}

Applying Proposition~\ref{prop:free-product-amplification} to the witnesses from Theorem~\ref{thm:main-coarse} and Corollary~\ref{cor:lcsc},
we obtain $2^{\aleph_0}$ pairwise non-quasi-isometric witnesses in each setting.

\end{document}